\documentclass[12pt]{amsart}
\usepackage{amsmath,amsxtra,amssymb,latexsym, amscd,amsthm,enumerate}
\usepackage{graphicx, tikz}
\topmargin = - 0.6cm
\numberwithin{equation}{section}

\newtheorem{thm}{\bf Theorem}[section]
\newtheorem{lem}[thm]{\bf Lemma}

\newtheorem{cor}[thm]{\bf Corollary}
\newtheorem{prop}[thm]{\bf Proposition}

\theoremstyle{definition}
\newtheorem{ex}[thm]{Example}
\newtheorem{rem}[thm]{Remark}

\newtheorem*{thm*}{Theorem}

\def\NN{{\mathbb N}}
\def\PP{{\mathbb P}}
\def\mm{{\mathfrak m}}

\def\pp{{\mathfrak p}}
\def\a{{\alpha}}
\def\b{{\beta}}

\def\e{{\varepsilon}}

\DeclareMathOperator{\reg}{reg}

\DeclareMathOperator{\height}{ht}
\DeclareMathOperator{\depth}{depth}

\begin{document}

\title{Buchsbaumness and Castelnuovo-Mumford regularity of non-smooth monomial curves}

\author{Tran Thi Gia Lam}
\address{Phu Yen University, Tuy Hoa, Vietnam}
\email{tranthigialam@pyu.edu.vn}

\author[Trung]{Ngo Viet Trung}
\address{International Centre for Research and Postgraduate Training\\ Institute of Mathematics \\ Vietnam Academy of Science and Technology, Hanoi, Vietnam}
\email{nvtrung@math.ac.vn}



\begin{abstract}
Projective monomial curves correspond to rings generated by monomials of the same degree in two variables. Such rings always have finite Macaulayfication. We show how to characterize the Buchsbaumness and the Castelnuovo-Mumford regularity of these rings by means of their finite Macaulayfication, and we use this method to study the Buchsbaumness and to estimate the Castelnuovo-Mumford regularity of large classes of non-smooth monomial curves in terms of the given monomials.
\end{abstract}

\maketitle


\section*{Introduction}

Let $k[x,y]$ be a polynomial ring over a field $k$ in two variables $x,y$.
Let $R = k[M]$ be the subring of $k[x,y]$ generated by a set $M$ of monomials of a given degree $d$. 
Then $R$ is the homogenous coordinate ring of the (projective) monomial curve given parametrically by $M$. This class of rings often serves as basic objects for several problems. For instance,  the first non-trivial example of a non-Cohen-Macaulay ring is $k[x^4,x^3y,xy^3,y^4]$, which was found by Macaulay \cite{Ma}. In turned out that this ring is also the first non-trivial Buchsbaum ring, a generalization of Cohen-Macaulay ring \cite{BSV,SV}. It is known that $R$ is a Cohen-Macaulay or Buchsbaum ring if and only if the first local cohomology module of $R$, which is the Hartshorne-Rao module of the corresponding curve, vanishes or a vector space over $k$.

Suppose that $M = \{x^d, x^{\a_1}y^{d-\a_1},...,x^{\a_n}y^{d-\a_n},y^d\}$. It is natural to ask whether there is a characterization of the Buchsbaumness of $R$ in terms of the sequence $d,\a_1,...,\a_n$. 
In general, this problem is difficult because of the many involved parameters.
Surprisingly, if $x^{d-1}y, xy^{d-1} \in M$, one can show that $R$ is a Buchsbaum ring if and only if $d,\a_1,...,\a_n$ satisfy a system of linear inequalities \cite{Tr1}. This case is of geometric interest because the monomial curve given parametrically by $M$ is smooth if and only if $x^{d-1}y, xy^{d-1} \in M$. On the other hand, there are very few known non-Cohen-Macaulay Buchsbaum rings $R$ whose corresponding monomial curves are not smooth \cite{Br,BSV,Tr1}.

Another problem is to estimate the Castelnuovo-Mumford regularity $\reg(R)$ in terms of the sequence $d,\a_1,...,\a_n$. This problem is of great interest because $\reg(R)$ controls the shifts of the graded minimal free resolution of $R$ \cite{EG}. By the Eisenbud-Goto conjecture \cite{EG}, which was proved for projective curves by Gruson, Lazarsfeld and Peskine \cite{GLP}, we know that $\reg(R) \le d-n$. It is easy to see that $d-n-1$ is the sum of the integer gaps of the sequence $\a_1,...,\a_n$. A much better bound was given by L'vovsky \cite{Lv} in terms of the sum of the largest and second largest integer gaps.
For smooth curves, Hellus, Hoa and St\"uckrad  \cite{HHS} showed that $\reg(R)$ is bounded even by a fraction of the largest integer gap (see Section 3 for details). One can also find explicit regularity formulas for larges classes of smooth monomial curves \cite{HHS,Ni1}. To prove similar results for non-smooth monomial curves seems to be a daunting task because it was even an open problem to give a combinatorial proof of the Eisenbud-Goto conjecture for monomial curves \cite{BGT}.  This problem was recently solved by Nitsche \cite{Ni2}. 
\par

We shall see that the root of the aforementioned results for smooth monomial curves lies in the fact that the Veronese subring of $k[x,y]$ generated by the monomials of degree $d$ is a finite Macaulayfication of $R$ if $x^{d-1}y, xy^{d-1} \in M$. We call a ring extension $R^*$ of $R$ in $k[x,y]$ a {\em finite Macaulayfication} if $R^*$ is a Cohen-Macaulay ring and the quotient $R^*/R$ has finite length. 
For an arbitrary set $M$ of monomials, $R$ always has a unique finite Macaulayfication $R^*$. In the first section of this paper we show that the Buchsbaumness and the Castelnuovo-Mumford regularity of $R$ can be characterized by means of $R^*$. If $R^*$ is a ring generated by a set $N$ monomials of degree $d$, these characterizations can be expressed in terms of the numerical semigroups generated by the first exponents of the monomials of $M$ and $N$ (Theorem \ref{criterion}). This gives an efficient method to study projective monomial curves. In the second and third sections of this paper we use 
this method to study the Buchsbaumness and to estimate the Castelnuovo-Mumford regularity of large classes of non-smooth monomial curves in terms of $M$.

Given any set $M$ of monomials of degree $d$ in $k[x,y]$  with $x^d,y^d \in M$, we can find a unique sequence of integers  $0=a_0 \le a_1 \le \cdots \le a_{2r+1}=d$ with $a_{2i-1} < a_{2i}-1$, $i = 1,...,r$, such that 
$$M = \big\{x^\a y^{d-\a}|\ \a \in \bigcup_{i=0}^r[a_{2i}, a_{2i+1}]\big\},$$
where $[a_{2i}, a_{2i+1}]$ denotes the interval of integers $\a$, $a_{2i} \le \a \le a_{2i+1}$ and $r$ is the number of the integer gaps between these intervals. 
Note that $x^{d-1}y, xy^{d-1} \in M$ means $a_1 > 0$ and $a_{2r} < d$.
We will investigate non-smooth monomial curves of the following types:\par
{\bf Type  A}: $a_1 = 0$, $a_{2r} <  d$, and $2a_2 -1 \le a_3$, $r \ge 1$. \par
{\bf Type B}: $a_1 = 0$,  $a_{2r} = d$, and $2a_2 -1 \le a_3$,  $a_{2r-2} +d-1 \le 2a_{2r-1}$, $r \ge 2$. \par
\noindent By the duality of the variables $x,y$, Type A also covers the case $a_1 > 0$, $a_{2r} =  d$, and $a_{2r-2} +d-1 \le 2a_{2r-1}$. These types represent large classes of non-smooth monomial curves; see the various examples in Sections 2 and 3. 

If $r=1$ in Type A or $r=2$ in Type B, then $R$ is a Cohen-Macaulay ring \cite{Tr1} and $\reg(R)$ was computed in \cite{La}. If $r \ge 2$ in Type A or $r \ge 3$ in Type B, then $R$ is not a Cohen-Macaulay ring.
However, the finite Macaulayfication of $R$ can be computed and it has a relatively simple structure.
Using this fact, we show that $R$ is a Buchsbaum ring if and only if $a_1,...,a_{2r+1}$ satisfy an explicit system of linear inequalities if and only if $\reg(R) = 2$ (Theorems \ref{BmA} and \ref{BmB}). The first condition provides a practical way to test the Buchsbaumness, while the second condition implies that the corresponding monomial curve is defined by equations of degree at most 3.

For all curves of Types A and B, we are able to show that there are bounds for $\reg(R)$ in terms of a fraction of the largest integer gap as in the aforementioned result of Hellus, Hoa, and St\"uckrad (Theorem \ref{regA} and Theorem \ref{regB}). These bounds are so close to $\reg(R)$ that we can derive regularity formulas for the cases $r = 2$ of Type A and $r = 3$ of Type B.  
\smallskip

\noindent {\bf Theorem \ref{regA1}.}
{\em Let $1 < a <  b <  c  < d$ be a sequence of integers with $b < c-1$ and  
$R = k\big[x^\a y^{d-\a}|\ \a \in \{0\} \cup [a,b] \cup [c,d]\big]$. Assume that $2a -1 \le b$. 
Let $\e = \min\{b,d-c\}$. Then $\reg(R) = \left \lfloor \dfrac{c-b-2}{\e}\right \rfloor +2.$}
\smallskip

\noindent {\bf Theorem \ref{regB1}}.
{\em Let $1 < a < b < c < e < d-1$ be a sequence of integers with  $b < c-1$ and
$R = k\big[x^\a y^{d-\a}|\ \a \in \{0,d\} \cup [a,b] \cup [c,e] \big]$. 
Assume that $2a -1 \le b$ and $c+d-1 \le 2e$.
Let $\e = \min\{b,d-c\}$. Then
\begin{enumerate}[\indent \rm (1)]
\item $\left \lfloor \dfrac{c-b-2}{\e}\right \rfloor + 2 \le \reg(R) \le \left \lfloor \dfrac{c-b-2}{\e}\right \rfloor + 3$.
\item $\reg(R) = \left \lfloor \dfrac{c-b-2}{\e}\right \rfloor + 2$ if one of the following conditions is satisfied:\par
{\rm (i)} $a-1 \le e-c$ and $d-e-1 \le b-a$, \par
{\rm (ii)}  $b-a \ge d-c$,\par
{\rm (iii)}  $e-c \ge b$.
\end{enumerate}}

The above results are remarkable because they give the first explicit regularity formulas for {\em large classes} of non-smooth monomial curves which are not arithmetically Cohen-Macaulay.
We could find in the literature only  explicit regularity formulas for particular cases of such curves which are curves in ${\Bbb P}^3_k$ \cite{BCFH} or associated to generalized arithmetic sequences \cite{BGG}. Explicit regularity formulas for non-smooth monomial curves which are arithmetically Cohen-Macaulay can be found in \cite{BGG,JS,La}.

In general, we can always characterize the Buchsbaumness and estimate $\reg(R)$ if we know the finite Macaulayfication of $R$. So our approach could be used to study other monomial curves as well. 


\section{Finite Macaulayfication}

Let $M$ be an arbitrary set of monomials in two variables $x$ and $y$ such that $x^d,y^d \in M$ for some integer $d \ge 0$.  Let $\mm$ be maximal monomial ideal of $k[M]$. For any module $S$ we denote by 
$H_\mm^i(S)$  the $i$-th local cohomology module of $S$ with respect to $\mm$. It is well known that $k[M]$ is a Cohen-Macaulay ring if and only if $H_\mm^1(k[M]) = 0$. Note that $\dim k[M] = 2$ and $H_\mm^0(k[M]) = 0$ because $k[M]$ is a domain. \par 

Let $E$ denote the set of the points $(\a,\b)$ such that $x^\a y^\b \in k[M]$.
Then $E$ is an additive semigroup in $\NN^2$ and $k[M]$ is the semigroup ring of $E$. We call $E$ the {\em affine semigroup} of $M$ or $k[M]$.
The Cohen-Macaulayness of affine semigroup rings in $\NN^d$ for arbitrary $d$ was studied
thoroughly in \cite{GSW, GW, HT}.

Given two subsets $A, B$ of $\NN$ or $\NN^2$, we define
$$A \pm B := \{u \pm v|\ u \in A, v \in B\}.$$
For $n \ge 1$, $nA$ denotes the sum of $n$ copies of $A$ (not the set $\{nu|\ u \in A\}$).

Let $E_1$ and $E_2$ denote the sets of the elements $(\a,\b) \in E$ with $\b = 0$ or $\a = 0$, respectively. Note that $e_1 := (d,0) \in E_1$ and $e_2 := (0,d) \in E_2$ because $x^d,y^d \in M$.
Define $$E^* := (E-E_1) \cap (E-E_2).$$
Then $E^*$ is an additive semigroup in $\NN^2$. Let 
$$M^* := \{x^\a y^\b|\ (\a,\b) \in E^*\}.$$
It is easy to see that $M \subseteq M^*$ and $(M^*)^* = M^*$.

\begin{thm} \label{CM} 
Let $M$ and $M^*$ be as above. Then\par
{\rm (i)} $k[M^*]$ is a Cohen-Macaulay ring.\par
{\rm (ii)} $k[M^*]/k[M]$ has finite length.\par
{\rm (iii)} $k[M^*] = k[M]$ if $k[M]$ is a Cohen-Macaulay ring,\par 
{\rm (iv)} $k[M^*]$ is the unique Cohen-Macaulay ring containing $k[M]$ in $k[x,y]$ such that $k[M^*]/k[M]$ has finite length.\par
{\rm (v)} $H_\mm^1(k[M]) \cong k[M^*]/k[M]$ and $H_\mm^2(k[M]) \cong H_\mm^2(k[M^*])$.
\end{thm}

\begin{proof}
(i) Let $\mm^*$ be the maximal bigraded ideal of $k[M^*]$. By \cite[Corollary 3.4(i)]{HT}, we have $H_{\mm^*}^1(k[M^*]) = 0$. Hence, $k[M^*]$ is a Cohen-Macaulay ring. \par

(ii) Since $k[M^*]$ is a domain, $k[M^*]_\pp$ is Cohen-Macaulay for all primes $\pp$ of $k[M^*]$ with $\height \pp = 1$. By \cite[Corollary 2.4]{HT}, this implies 
$$\dim k[M^*]/k[M] < \dim k[M]-1 = 1.$$ 
Hence, $\dim k[M^*]/k[M] = 0.$ Therefore, $k[M^*]/k[M]$ has finite length. \par

(iii) If $k[M]$ is a Cohen-Macaulay ring, $k[M^*] = k[M]$ by \cite[Corollary 2.2]{HT}. \par

(iv) Let $N \supseteq M$ be a set of monomials in $k[x,y]$ such that $k[N]$ is a Cohen-Macaulay ring and $k[N]/k[M]$ is of finite length. 
As we have seen above,  $k[N^*] = k[N]$. By definition,  $N^* \supseteq M^*$. This implies $k[N] \supseteq k[M^*]$.
Since $k[N]/k[M]$ is of finite length, there exists a number $n$ such that $\mm^n k[N] \subseteq k[M]$.
Let $x^\a y^\b$ be an arbitrary monomial of $N$. Then $x_i^{nd}x^\a y^\b \in M$, $i = 1,2$.
From this it follows that $(a,b) \in (E - ne_1) \cap (E-ne_2) \subseteq E^*$. 
Therefore, $N \subseteq M^*$, which implies $k[N] \subseteq k[M^*]$. So we have $k[N] = k[M^*]$. \par

(v) Consider the derived exact sequence of local cohomology of the modules of the short exact sequence
$$0 \to k[M] \to k[M^*] \to  k[M^*] /k[M]  \to 0.$$
Since $k[M^*] $ is a Cohen-Macaulay module over $k[M^*] $, $H_\mm^i(k[M^*]) = 0$ for $i \ne 2$.
Since $k[M^*]/k[M] $ is a finite-dimensional vector space, $H_\mm^0(k[M^*]/k[M]) \cong k[M^*]/k[M]$ and $H_\mm^i(k[M^*]/k[M]) = 0$ for $i > 0$. Therefore, $H_\mm^1(k[M]) \cong k[M^*]/k[M]$ 
and $H_\mm^2(k[M]) \cong H_\mm^2(k[M^*])$.
\end{proof}
\smallskip

\begin{rem}~
\begin{enumerate}
\item By Theorem \ref{CM}(i) and (ii), $k[M^*]$ is a finite Macaulayfication of $k[M]$. By Theorem \ref{CM}(iv), it is unique.
\item By Theorem \ref{CM}(ii), $E^* \setminus E$ is a finite set because the vector space $k[M^*]/k[M]$ has a basis consisting of the monomials of $k[M^*] \setminus k[M]$, whose exponents correspond to the elements of $E^* \setminus E$.  
\item By Theorem \ref{CM}(iii),  $k[M]$ is a Cohen-Macaulay ring if and only if 
$E^* = E.$ This fact can be also deduced from another criterion of Goto, Suzuki and Watanabe 
\cite[Theorem 2.6]{GSW}, which replaces the condition $E^* = E$ by the weaker condition $(E - e_1) \cap (E - e_2) = E.$  Alternative Cohen-Macaulay criteria for $k[M]$ can be found in \cite{FH, HS}.
\end{enumerate}
\end{rem}

\begin{rem}
We call a ring extension $S$ of a local ring $(R,\mm)$ a finite Macaulayfication of $R$ if $S$ is a Cohen-Macaulay ring and the quotient $S/R$ has finite length.
Any ring $R$ with $\dim R > \depth R \ge 2$ does not have a finite Macaulayfication. If such a ring $R$ has a finite Macaulayfication $S$, then from the exact sequence $0 \to R \to S \to S/R \to 0$ we can derive that $H_\mm^1(R) = \ell(S/R) \neq 0$. This implies $\depth R = 1$. 
\end{rem}

We can compute the finite Macaulayfication $k[M^*]$ indirectly as follows.

\begin{cor} \label{sufficient}  
Assume that $M$ is contained in a set $N \subseteq M^*$ such that $k[N]$ is a Cohen-Macaulay ring.
Then $k[M^*] = k[N]$.
\end{cor}

\begin{proof}
We have $k[N]/k[M] \subseteq k[M^*]/k[M]$. Since $k[M^*]/k[M]$ is of finite length, $k[N]/k[M]$ is of finite length, too. Therefore, $k[M^*] = k[N]$ by Theorem \ref{CM}(iv).
\end{proof}

\begin{ex} \label{Veronese} 
Let $M$ be a set of monomials of degree $d$ with $x^{d-1}y,xy^{d-1} \in M$. Then $(d-1,1), (1,d-1) \in E$.
Let $N$ be the set of all monomials of degree $d$ in $k[x,y]$. For all monomials $x^\a y^{d-\a} \in N$, we have
\begin{align*}
(\a,d-\a) + (d-\a-1)e_1 & = (d-\a)(d-1,1) \in E,\\
(\a,d-\a) + (\a-1)e_2 & = \a(1,d-1) \in E.
\end{align*}
Hence, $(\a,d-\a) \in (E-E_1) \cap (E-E_2) = E^*$.
This shows that $N \subseteq M^*$. Since $k[N]$ is a Veronese subring of $k[x,y]$, $k[N]$ is a Cohen-Macaulay ring. Therefore, $k[M^*] = k[N]$ by Corollary \ref{sufficient}.
\end{ex}

If we know the finite Macaulayfication of $k[M]$, we can use it to  test the Buchsbaumness of $k[M]$. It is known that $k[M]$ is a Buchsbaum ring if and only if $\mm H_\mm^1(k[M]) = 0$.
We refer the reader to \cite{SV} for the definition and properties of Buchsbaum rings, which are a natural generalization of Cohen-Macaulay rings. 

\begin{prop} \label{buchsbaum}
Let $E$ be the affine semigroup of $M$. Then $k[M]$ is a Buchsbaum ring if 
and only if $(E \setminus \{0\}) + E^* \subseteq E$. 
\end{prop}

\begin{proof}
The statement is a consequence of \cite[Lemma 4.11]{HT}.
It can be proved directly as follows.
By Theorem \ref{CM}(v), $\mm H_\mm^1(k[M]) = 0$ if and only if $\mm k[M^*] \subseteq k[M]$.
The last condition just means $(E \setminus \{0\}) + E^* \subseteq E$.
\end{proof}
\smallskip

\begin{rem}~
\begin{enumerate}
\item Proposition \ref{buchsbaum} was a special case of a general criterion for the Buchsbaumness of rings generated by monomials stated in \cite[Theorem 3.1]{Go}. However, this result does not hold in general and it was corrected by \cite[Lemma 4.11]{HT}. 
\item The condition $(E \setminus \{0\}) + E^* \subseteq E$ of Proposition \ref{buchsbaum} can be weakened to 
$(E \setminus \{0\}) + (E - 2e_1)\cap(E-2e_2) \subseteq E$ \cite[Lemma 3]{Tr}. Alternative Buchsbaum criteria for $k[M]$ can be found in \cite{FH}.
\end{enumerate}
\end{rem}

From now on, let $M$ be a set of monomials of the same degree $d$ with $x^d, y^d \in M$. Let $R = k[M]$. 
Then $R$ is a standard graded algebra with $\deg x^\a y^\b = (\a+\b)/d$ for all $x^\a y^\b \in R$.
For any graded $R$-module $S$ we denote by $S_n$ the $n$-th component of $S$.
For $i \ge 0$, let $a_i(S) := \max\{n|\  H_\mm^i(S)_n \neq 0\}$, where $a_i(S) = -\infty$ if $H_\mm^i(S) = 0$. The {\em Castelnuovo-Mumford regularity} of $S$ is defined by
$\reg(S) := \max\{a_i(S)+i|\ i \ge 0\}.$ 
It is a measure for the complexity of the graded structure of $S$ (see e.g. \cite{EG,Tr3}). Since $H_\mm^i(R) = 0$ for $i = 0$ and $i >  2$, we have 
$\reg(R) = \max\{a_1(R)+1,a_2(R)+2\}.$
\par

Let $Q := (x_1^d,x_2^d)$. Then $Q$ is a parameter ideal of $R$ generated by linear forms. 
From this it follows that there exists a number $n$ such that $R_{n+1} =  Q_{n+1}$.  
We call the least number $n$ with this property the {\em reduction number} of $R$ with respect to $Q$, denoted by $r_Q(R)$. This number can be computed easily.
By \cite[Proposition 3.2]{Tr2}, we have $a_2(R)+2 \le r_Q(R) \le \reg(R)$. Therefore,
$$\reg(R) = \max\{a_1(R)+1, r_Q(R)\}.$$
Note that the reduction number of $R$ with respect to $Q$ is defined one less in \cite{Tr2}. \par

Let $R^* = k[M^*]$. 
Since $R^*/R$ is a finite-dimensional vector space, $R^*$ is a finite graded $R$-module. Therefore, there also exists  a number $n$ such that $R^*_{n+1} =  (QR^*)_{n+1}$. We call the least number $n$ with this property the reduction number of $R^*$ with respect to $Q$, denoted by $r_Q(R^*)$. Since $R^*$ is a Cohen-Macaulay ring, $H_\mm^i(R^*) = 0$ for $i \neq 2$. Therefore, $a_i(R^*) = -\infty$ for $i \neq 2$. 
By a module-theoretic version of \cite[Proposition 3.2]{Tr2},  this implies
$$\reg(R^*) = r_Q(R^*) = a_2(R^*) + 2.$$

We can characterize $\reg(R)$ by means of $R^*$ as follows.
Let $a(R^*/R)$ denote the largest number $n$ such that $(R^*/R)_n \neq 0$. If $(R^*/R) = 0$, we set
$a(R^*/R) = -\infty$.

\begin{prop} \label{reg}
Let $R$, $R^*$ and $Q$ be as above. Then
$$\reg(R) = \max\{a(R^*/R)+1,r_Q(R^*)\}.$$
\end{prop}

\begin{proof}
By Theorem \ref{CM}(v),  $a_1(R) =  a(R^*/R)$ and $a_2(R) = a_2(R^*)$.
Therefore, 
\begin{align*}
\reg(R) & = \max\{a_1(R)+1,a_2(R)+2\}\\
 & = \max\{a(R^*/R)+1,a_2(R^*)+2\} = \max\{a(R^*/R)+1,r_Q(R^*)\}.
 \end{align*}
\end{proof}

We always have $r_Q(R^*) \le r_Q(R)$. In fact, $r_Q(R^*) = a_2(R)+2$ by the proof of Theorem \ref{reg} and $a_2(R)+2 \le r_Q(R)$ by \cite[Proposition 3.2]{Tr2}.  The following example shows that $r_Q(R)$ can be arbitrarily larger than $r_Q(R^*)$.

\begin{ex} 
If $x^{d-1}y,xy^{d-1} \in M$, then $R^*$ is the Veronese subring of $k[x,y]$ generated by the monomials of degree $d$. It is easy to see that $r_Q(R^*) = 1$. By \cite[Proposition 4.1]{La}, $r_Q(R)$ can be any number between 1 and $d-2$. 
\end{ex}

\begin{cor}
Let $R$, $R^*$ and $Q$ be as above. Assume that $R$ is a Buchsbaum ring. Then
$$\reg(R) = r_Q(R) \in \{r_Q(R^*), r_Q(R^*)+1\}.$$
\end{cor}

\begin{proof}
Since $R$ is a Buchsbaum ring, $\reg(R) = r_Q(R)$ by \cite[Corollary 3.5]{Tr2}. 
By Proposition \ref{reg}, to prove that $\reg(R) \in \{r_Q(R^*), r_Q(R^*)+1\}$ we only need to show that $a(R^*/R)  \le r_Q(R^*)$.  
For $n \ge r_Q(R^*)$, we have $R^*_{n+1} = (QR^*)_{n+1}$. 
By Proposition \ref{buchsbaum}, $\{e_1,e_2\} + E^* \subseteq E$, 
which implies $(QR^*)_{n+1} = R_{n+1}$. 
Therefore, $(R^*/R)_{n+1} = 0$. Hence, $a(R^*/R)  \le r_Q(R^*)$.
\end{proof}

It is easy to find a Buchsbaum ring $R$ with $\reg(R) = r_Q(R^*) + 1$.  

\begin{ex}
Let $R = k[x^4,x^3y,xy^3,y^4]$. Then  $R^* = k[x^4,x^3y,x^2y^2, xy^3,y^4]$ with $r_Q(R^*) = 1$. 
Since $R$ is a Buchsbaum ring with $r_Q(R)= 2$, we have $\reg(R) = 2$.
\end{ex}

To check the Buchsbaumness or to estimate the regularity of $R$ we only need to work with sequences of integers if the finite Macaulayfication $R^*$ is generated by monomials of degree $d$. For that purpose we set 
$$G_M := \{\a|\ x^\a y^{d-\a}\in M\}$$
for every set $M$ of monomials of degree $d$ in $k[x,y]$. 

\begin{thm} \label{criterion}
Let $R = k[M]$. Assume that $R^* = k[N]$ where $N$ is a set of monomials of degree $d$. Let $Q = (x^d,y^d)$. Then 
\begin{enumerate}[\indent \rm (1)]
\item $R$ is a Buchsbaum ring if and only if $G_M + G_N = 2G_M$. 
\item $\reg(R) = \min\{n \ge r_Q(R^*)|\  nG_M = nG_N\}$.
\end{enumerate}
\end{thm}

\begin{proof}
Let $E$ be the affine semigroup of $M$. By Proposition \ref{buchsbaum}, $R$ is a Buchsbaum ring if and only if $(E \setminus \{0\}) + E^* \subseteq E$. Let $A = \{(\a,d-\a)|\ \a \in G_M\}$ and $B = \{(\a,d-\a)|\ \a \in G_N\}$. The assumption implies that $E$ and $E^*$ are generated by $A$ and $B$, respectively. 
Therefore, $(E \setminus \{0\}) + E^* \subseteq E$ if and only if $A + B \subseteq 2A$. This condition is satisfied if and only if $G_M + G_N \subseteq 2G_M$ or, equivalently, $G_M + G_N = 2G_M$.

By Proposition \ref{reg}, $\reg(R) = \max\{a(R^*/R)+1,r_Q(R^*)\}.$ We have $a(R^*/R) = \max\{n|\ R_n \neq R^*_n\}.$
Since $R$ and $R^*$ are generated by monomials of degree $d$, $R_n = k[x^\a y^{nd-\a}|\ \a \in nG_M]$ and $R^*_n  = k[x^\a y^{nd-\a}|\ \a \in nG_N].$ Therefore,
$$a(R^*/R) = \max\{n|\ nG_M \neq nG_N\}.$$
For $n \ge r_Q(R^*)$, $(n+1)G_N = \{0,d\} + nG_N \subseteq G_M+nG_N$. This implies $(n+1)G_M = (n+1)G_N$ if $nG_M = nG_N$. Thus, $a(R^*/R) + 1 = \min\{n|\ nG_M = nG_N\}$. 
From this it follows that $\reg(R) = \min\{n\ge r_Q(R^*)|\ nG_M =nG_N\}.$
\end{proof}

The reduction number $r_Q(R^*)$ can be also computed by means of $G_N$. 
Since $r_Q(R^*) = \min\{n|\ R^*_{n+1} = (QR^*)_{n+1}\}$, we have
$$r_Q(R^*) = \min\{n|\ (n+1)G_N = \{0,d\}+nG_N\}.$$


\section{Buchsbaumness}

Let $M$ be a set of monomials of degree $d$ in two variables $x$ and $y$ such that $x^d,y^d \in M$ for some integer $d \ge 0$. Then we can find a unique sequence of integers 
$0=a_0 \le  a_1 \le \cdots \le a_{2r+1} = d$
with $a_{2i-1} < a_{2i}-1$, $i = 1,...,r$, such that the set of the exponents $\a$, $x^\a y^{d-\a} \in M$, is given by
$$G_M = \bigcup_{i=0}^r[a_{2i}, a_{2i+1}],$$
where $[a_{2i}, a_{2i+1}]$ denotes the interval of integers $\a$, $a_{2i} \le \a \le a_{2i+1}$. 
The condition $a_{2i-1} < a_{2i}-1$ means that $[a_{2i-1}+1, a_{2i}-1]$ is an integer gap between the integer intervals of $G_M$.

Let $R = k[M]$. The aim of this section is to give criteria for $R$ to be a Buchsbaum ring in terms of the sequence $a_1,...,a_{2r+1}$. Let $R^*$ denote the finite Macaulayfication of $R$. We will concentrate on the case 
$R^* = k[N]$ where $N$ is a set of monomials of degree $d$ in $k[x,y]$.
In this case, $R$ is a Buchsbaum ring if and only if  $G_M + G_N \subseteq 2G_M$ by  Theorem \ref{criterion}(1).  
We can represent $G_M + G_N$ as a union of disjoint integer intervals in $[0,2d]$. 
Given such an interval $[u,v]$, we have to find conditions for $[u,v] \subseteq 2G_M$ in terms of the sequence $a_0,a_1,...,a_{2r+1}$. The following result shows that
$[u,v] \subseteq 2G_M$ if and only if $a_1,...,a_{2r+1}$ satisfies a system of linear inequalities.

Let $I = \{(m,n) \in \NN^2|\ 0 \le m,n \le r\}$.
For $(m,n), (m',n') \in I$, we define $(m',n') \le  (m,n)$ if $m' \le  m, n' \le n$. This gives a partial order on $I$. 
A subset $J$ of $I$ is {\em symmetric} if whenever $(m,n) \in J$, then $(n,m) \in J$. 
We call $J$ a {\em poset ideal} if whenever $(m',n') \le (m,n) \in J$, then $(m',n') \in J$. 
Let $J_{\max}$ resp. $J_{\min}$ denote the set of maximal resp. minimal elements of $J$.

\begin{lem} \label{cover}  
Let $0 \le u \le v \le 2d$ be two arbitrary integers. Let $(m,n)$ 
and $(m',n')$ be elements of $I$ such that $a_{2m+1}+a_{2n+1}$ is the maximum of all values 
$a_{2i+1}+a_{2j+1} < u$ and $a_{2m'}+a_{2n'}$ is the minimum of all values $a_{2i}+a_{2j} > v$, $(i,j) \in I$.
Then $[u,v] \subseteq 2G_M$ if and only if for every symmetric poset ideal $J$ of $I$ with $(m,n) \in J$ and $(m',n') \not\in J$, 
\begin{equation}
\max\big\{a_{2i+1}+a_{2j+1}|\ (i,j) \in J_{\max}\big\} +1 \ge 
\min\big\{a_{2i}+a_{2j}|\ (i,j) \in (I\setminus J)_{\min}\big\}.
\end{equation} 
\end{lem}

\begin{proof} 
Let $J$ be an arbitrary symmetric poset ideal of $I$ with $(m,n) \in J$ and $(m',n') \not\in J$. Set
\begin{align*}
a & = \max\{a_{2i+1}+a_{2j+1}|\ (i,j) \in J_{\max}\},\\
b & =  \min\{a_{2i}+a_{2j}|\ (i,j) \in (I\setminus J)_{\min}\}.
\end{align*} 
If $a +1 < b$, 
$$[a+1,b-1] \subseteq [a_{2m+1}+a_{2n+1}+1, a_{2m'}+a_{2n'}-1] \subseteq [u,v].$$
Let $c$ be an element of $[a+1,b-1]$. Since $a < c$, $c \not\in [a_{2i}+a_{2j}, a_{2i+1}+a_{2j+1}]$ for any $(i,j) \in J$. Since  $c < b$, $c \not\in [a_{2i}+a_{2j}, a_{2i+1}+a_{2j+1}]$ for any $(i,j) \in I \setminus J$. 
Note that $2G_M = \bigcup_{(i,j) \in I}[a_{2i}+a_{2j},a_{2i+1} + a_{2j+1}]$. Then $c \not\in 2G_M$. Therefore, $[u,v]\not\subseteq 2G_M$. \par
Conversely, if $[u, v] \not\subseteq 2G_M$, 
choose an element $c \in [u,v]$ such that $c \not\in 2G_M$.  Let $J$ be the set of all $(i,j) \in I$ such that $c > a_{2i+1} + a_{2j+1}$. Then $J$ is a symmetric poset ideal of $I$ with $(m,n) \in J$ and $(m',n') \not\in J$. If we define $a$ as above, then $a < c$. For $(i,j) \in I\setminus J$, we have $c \le a_{2i+1}+a_{2j+1}$. Since $c \not\in [a_{2i}+a_{2j},a_{2i+1}+a_{2j+1}] \subset 2G_M$, this implies $c < a_{2i}+a_{2j}$. If we define $b$ as above, then $c < b$. Hence, $a + 1 < b$.
\end{proof}

\begin{ex} \label{smooth}
Let us consider the smooth case $a_1 > 0$ and $a_{2r} <d$. In this case, $R^* = k[N]$, where $N$ is the set of all monomials of degree $d$ in $k[x,y]$. 
Since $G_N = [0,d]$, $G_M + G_N \supseteq \{0,d\} + [0,d] = [0,2d]$. Hence $G_M + G_N = [0,2d]$. 
By Theorem \ref{criterion}(1), $R$ is a Buchsbaum ring if and only if $[0,2d] = 2G_M$.
Since $[0,2a_1], [2a_{2r},2d] \subseteq 2G_M$, $[0,2d] = 2G_M$ if and only if $[2a_1+1,2a_{2r}-1] \subseteq 2G_M$. Note that $a_1+a_1$ is the maximum of all values 
$a_{2i+1}+a_{2j+1} < 2a_1+1$ and $a_{2r}+a_{2r}$ is the minimum of all values $a_{2i}+a_{2j} > 2a_{2r}-1$, $(i,j) \in I$. Then $[2a_1+1,2a_{2r}-1] \subseteq 2G_M$ if and only if the inequality (2.1) holds for every symmetric poset ideal $J$ of $I$ with $(r,r) \not\in J$. That is the Buchsbaum criterion for
the smooth case given in \cite[Theorem 4.7]{Tr1}. 
\end{ex}

Our approach also yields the following interesting relationship between Buchsbaumness and regularity in the smooth case.

\begin{thm} \label{Bm}
Let $0=a_0 \le  a_1 \le \cdots \le a_{2r+1} = d$ be a sequence of integers with $a_{2i-1} < a_{2i}-1$, $i = 1,...,r$, and $R = k\big[x^\a y^{d-\a}|\ \a \in \bigcup_{i=0}^r[a_{2i}, a_{2i+1}\big]$. 
Assume that $a_1 > 0$ and $a_{2r} < d$.
Then $R$ is a Buchsbaum ring if and only if $\reg(R) = 2$.
\end{thm}

\begin{proof}
Let $M = \{x^\a y^{d-\a}|\ \a \in \bigcup_{i=0}^r[a_{2i}, a_{2i+1}]\}$ and $N = \{x^{\a}y^{d-\a}| \ \a \in  [0,d]\}$. Then $R = k[M]$ and $R^* = k[N]$. As we have seen in Example \ref{smooth}, $G_M + G_N = [0,2d] = 2G_N$. Therefore, $R$ is a Buchsbaum ring if and only if $2G_N = 2G_M$. 
Let $Q = (x^d,y^d)$. Then $r_Q(R^*) = 1$. 
By Theorem \ref{criterion}(2), this implies $\reg(R) = \min\{n \ge 1|\ nG_N = nG_M\}$. Since $G_N \neq G_M$, $R$ is a Buchsbaum ring if and only if $\reg(R)=2$.
\end{proof}

Now we will study the Buchsbaumness of non-smooth curves of Type A. 

\begin{thm} \label{BmA}
Let $0=a_0 \le  a_1 \le \cdots \le a_{2r+1} = d$ be a sequence of integers with $a_{2i-1} < a_{2i}-1$, $i = 1,...,r$, and $R = k\big[x^\a y^{d-\a}|\ \a \in \bigcup_{i=0}^r[a_{2i}, a_{2i+1}]\big]$. 
Assume that $a_1 = 0$, $a_{2r} < d$, and $2a_2-1 \le a_3$, $r \ge 2$. 
Then the following conditions are equivalent:
\begin{enumerate}[\indent \rm (1)]
\item $R$ is a Buchsbaum ring.
\item The inequality (2.1) holds for every symmetric poset ideal $J$ of $I$ with $(1,1) \in J$ and $(r,r) \not\in J$
\item $\reg(R) = 2$.
\end{enumerate}
\end{thm}

\begin{proof}
Let $M = \{x^{\a}y^{d-\a}|\ a \in \bigcup_{i=0}^r[a_{2i}, a_{2i+1}]\}$ and $N = \{x^{\a}y^{d-\a}|\ \a \in \{0\} \cup [a_2,d]\}$. 
Then $M \subset N$ and $k[N]$ is a Cohen-Macaulay ring by \cite[Corollary 3.4]{Tr1}.
We have $N \setminus M \subseteq \{x^{\a}y^{d-\a}| \ \a \in [a_3+1,a_{2r}-1]\}$. 
For $\a \in [a_3+1,a_{2r}-1]$, write $\a = ta_2 + c$ with $0 \le c < a_2$.
Since $a_3+1 \ge 2a_2$, $t \ge 2$ and $a_2+c  \le a_3$. Hence, $(a_2+c,d-a_2-c) \in E$,
where $E$ is the affine semigroup of $M$. We have
\begin{align*}
(\a,d-\a) + (d-\a-1)e_1 & = (d-\a)(d-1,1) \in E,\\
(\a,d-\a) + (t-1)e_2 & = (t-1)(a_2,d-a_2) + (a_2+c,d-a_2-c) \in E.
\end{align*}
Hence $(\a,d-\a) \in (E-E_1) \cap (E-E_2) = E^*$. From this it follows that
$N \subseteq M^*$. Therefore, $R^* = k[N]$ by Corollary \ref{sufficient}. \par

We have $G_N = \{0\} \cup [a_2,d]$. Hence, $2G_N = \{0\} \cup [a_2,d] \cup [2a_2,2d]$.
Since $2a_2 - 1\le a_3 < d$, 
$$[a_2,d] \cup [2a_2,2d] = [a_2,2d] = [a_2,a_3] \cup [2a_2,2d].$$ 
Note that $[a_2,a_3] \subset G_M$ and $[2a_2,2d] = \{a_2,d\} + [a_2,d] \subset G_M+G_N$. 
Then $2G_N \subseteq G_M+G_N$. Since $G_M \subseteq G_N$, this implies $G_M+G_N = 2G_N$. 
By Theorem \ref{criterion}(1), $R$ is a Buchsbaum ring if and only if $2G_N = 2G_M$.
\par

On the other hand, we have $[a_2,2a_3] = [a_2,a_3] \cup [2a_2,2a_3] \subseteq 2G_M$ and $[2a_{2r},2d] \subseteq 2G_M$. Since $2G_N = \{0\} \cup [a_2,2d]$, this implies  
$$2G_N \setminus 2G_M \subseteq [a_2,2d] \setminus ([a_2,2a_3]\cup[2a_{2r},2d]) 
=  [2a_3+1,2a_{2r}-1].$$
Therefore, $2G_N = 2G_M$ if and only if $[2a_3+1,2a_{2r}-1] \subseteq 2G_M$. 

It is clear that $a_3+a_3$ is the maximum of all values 
$a_{2i+1}+a_{2j+1} < 2a_3+1$ and $a_{2r}+a_{2r}$ is the minimum of all values $a_{2i}+a_{2j} > 2a_{2r}-1$, $(i,j) \in I$. By Lemma \ref{cover}, $[2a_3+1,2a_{2r}-1] \subseteq 2G_M$ if and only if the inequality (2.1) holds for every symmetric poset ideal $J$ of $I$ with $(1,1) \in J$ and $(r,r) \not\in J$.
This shows the equivalence of the conditions (1) and (2).

Now we are going to prove the equivalence of the conditions (1) and (3).
Let $Q = (x^d,y^d)$. By \cite[Theorem 2.1]{La}, $r_Q(R^*) =  \left\lceil \dfrac {d-1}{d-a_2}\right\rceil.$
Since $2a_2 -1 \le a_3 < d$, $d-1 < 2(d-a_2)$. Hence $r_Q(R^*) = 2$. 
By Theorem \ref{criterion}(2), this implies 
$$\reg(R) =  \min\{n \ge 2|\ nG_N = nG_M\}.$$
We have shown above that $R$ is a Buchsbaum ring if and only if $2G_N = 2G_M$.
Therefore, $R$ is a Buchsbaum ring if and only if $\reg(R) = 2$.  
\end{proof}

One can easily write down the inequalities (2.1) in Theorem \ref{BmA}(2).
For instance, the case $r = 2$ yields the following concrete criterion.

\begin{cor} \label{BmA1}
Let $1 < a <  b <  c  < d$ be a sequence of integers with $b < c-1$ and  
$R = k\big[x^\a y^{d-\a}|\ \a \in \{0\} \cup [a,b] \cup [c,d]\big]$. Assume that $2a -1 \le b$.
Then $R$ is a Buchsbaum ring if and only if 
\begin{align*}
2b + 1 & \ge c,\\
 \max\{2b,d\} + 1 &  \ge a+c,\\
b + d +1 & \ge 2c. 
\end{align*}
\end{cor}

\begin{proof}
There are 3 different symmetric posets $J$ of $I$ with $(1,1) \in J$ and $(2,2) \not\in J$:
\begin{enumerate}
\item $\{(0,1),(1,0),(1,1)\}$,
\item $\{(0,1),(1,0),(0,2),(2,0),(1,1)\}$,
\item $\{(0,1),(1,0),(0,2),(2,0),(1,1),(1,2),(2,1)\}$.
\end{enumerate}
By (2.1), these posets yield the 3 inequalities in the statement.
\end{proof}

\begin{ex} \label{BmA2}
Let $R = k[x^d,x^{d-1}y,x^3y^{d-3},x^2y^{d-2},y^d]$, $d \ge 6$. Then $a = 2$, $b=3$, and $c = d-1$, which satisfy the assumption of Corollary \ref{BmA1}. Using Corollary \ref{BmA1} we can easily check that $R$ is a Buchsbaum ring if and only if $d = 6$.  
\end{ex}

For non-smooth curves of Type B we have the following criterion for the Buchsbaumness.

\begin{thm} \label{BmB}
Let $0=a_0 \le  a_1 \le \cdots \le a_{2r+1} = d$ be a sequence of integers with $a_{2i-1} < a_{2i}-1$, $i = 1,...,r$, and $R = k\big[x^\a y^{d-\a}|\ \a \in \bigcup_{i=0}^r[a_{2i}, a_{2i+1}]\big]$. 
Assume that $a_1 = 0$, $a_{2r} = d$, $2a_2 - 1 \le a_3$, and $a_{2r-2}+d-1 \le a_{2r-1}$, $r \ge 3$. 
Then the following conditions are equivalent:
\begin{enumerate}[\indent \rm (1)]
\item $R$ is a Buchsbaum ring.
\item The inequality (2.1) holds for every symmetric poset ideal $J$ of $I$ with $(1,1) \in J$ and $(r-1,r-1) \not\in J$.
\item $\reg(R) = 2$.
\end{enumerate}
\end{thm}

\begin{proof}
Let $M = \{x^{\a}y^{d-\a}|\ a \in \bigcup_{i=0}^r[a_{2i}, a_{2i+1}]\}$ and $N = \{x^{\a}y^{d-\a}| \ \a \in \{0,d\} \cup [a_2,a_{2r-1}]\}$. 
Then $M \subset N$ and $k[N]$ is a Cohen-Macaulay ring by \cite[Theorem 2.1]{Tr1}. 
We have $N \setminus M \subseteq \{x^{\a}y^{d-\a}| \ \a \in [a_3+1,a_{2r-2}-1]\}$. 

Let $E$ be the affine semigroup of $R = k[M]$.
Similarly as in the proof of Theorem \ref{BmA} we can use the assumptions $2a_2 - 1 \le a_3$ and $a_{2r-2} +d-1 \le 2a_{2r-1}$ to show that $(\a,d-\a) \in (E-E_1) \cap (E-E_2) = E^*$ for $\a \in [a_3+1,a_{2r-2}-1]$. Thus, $\{x^{\a}y^{d-\a}| \ \a \in [a_3+1,a_{2r-2}-1]\} \subset M^*$.
From this it follows that $N \subseteq M^*$. Therefore, $R^* = k[N]$ by Corollary \ref{sufficient}. \par

We have $G_N = \{0,d\} \cup [a_2,a_{2r-1}]$. Hence, 
$$2G_N = \{0,d,2d\} \cup [a_2,a_{2r-1}]  \cup [2a_2,2a_{2r-1}] \cup [a_2+d,a_{2r-1}+d].$$
Since $2a_2 - 1 \le a_3 < a_{2r-1}$ and $a_2 + d - 1 < a_{2r-2} +d-1 \le 2a_{2r-1}$, 
\begin{align*}
[a_2,a_{2r-1}] \cup [2a_2,2a_{2r-1}] & = [a_2,2a_{2r-1}],\\
 [2a_2,2a_{2r-1}] \cup [a_2+d,a_{2r-1}+d] & = [2a_2,a_{2r-1}+d].
\end{align*} 
Therefore,
\begin{align*}
2G_N & = \{0,2d\} \cup [a_2,a_{2r-1}+d]\\
& = \{0,2d\} \cup [a_2,a_3]  \cup [2a_2,2a_{2r-1}] \cup [a_{2r-2}+d,a_{2r-1}+d].
\end{align*}
Notice that $[a_2,a_3], [a_{2r-2}+d,a_{2r-1}+d] \subset 2G_M$ and
$$[2a_2,2a_{2r-1}] = \{a_2,a_{2r-1}\} + [a_2,a_{2r-1}]  \subset G_M+G_N.$$ 
Then $2G_N \subseteq G_M+G_N$. This implies $G_M+G_N = 2G_N$. 
It follows from Theorem \ref{criterion}(1) that $R$ is a Buchsbaum ring if and only if $2G_N = 2G_M$. \par

On the other hand, we have
\begin{align*}
[a_2,2a_3] & = [a_2,a_3] \cup [2a_2,2a_3] \subset 2G_M,\\
[2a_{2r-2},a_{2r-1}+d] & = [2a_{2r-2},2a_{2r-1}] \cup [a_{2r-2}+d,a_{2r-1}+d]  \subset 2G_M.
\end{align*}
Since $2G_N = \{0,2d\} \cup [a_2,a_{2r-1}+d]$, 
\begin{align*}
2G_N \setminus 2G_M & \subseteq [a_2,a_{2r-1}+d] \setminus ([a_2,2a_3] \cup [2a_{2r-2},a_{2r-1}+d]) \\                 & = [2a_3+1,2a_{2r-2}-1].
\end{align*}
Therefore, $2G_N = 2G_M$ if and only if $[2a_3+1,2a_{2r-2}-1] \subseteq 2G_M$. \par

It is clear that $a_3+a_3$ is the maximum of all values 
$a_{2i+1}+a_{2j+1} < 2a_3+1$ and $a_{2r-2}+a_{2r-2}$ is the minimum of all values $a_{2i}+a_{2j} > 2a_{2r-2}-1$, $(i,j) \in I$. By Lemma \ref{cover}, $[2a_3+1,2a_{2r-2}-1] \subseteq 2G_M$ if and only if 
the inequality (2.1) holds for every symmetric poset ideal $J$ of $I$ with $(1,1) \in J$ and $(r-1,r-1) \not\in J$. This shows the equivalence of the conditions (1) and (2).

Let $Q = (x^d,y^d) \subset R^*$. By \cite[Theorem 3.3(1)]{La}, $r_Q(R^*) =  \left\lceil \dfrac {a_2+d-1}{a_{2r-1}}\right\rceil.$
Since $2a_{2r-1} \ge a_{2r-2}+d-1 > a_2+d-1$, $r_Q(R^*) = 2$. By Theorem \ref{criterion}(2), this implies 
$$\reg(R) =  \min\{n \ge 2|\ nG_N = nG_M\}.$$ 
We have seen above that $R$ is a Buchsbaum ring if and only if $2G_N = 2G_M$.
Therefore, $R$ is a Buchsbaum ring if and only if $\reg(R) = 2$. This shows the equivalence of the conditions (1) and (3).
\end{proof}

For the case $r = 3$ of Type B we have the following concrete criterion.

\begin{cor} \label{BmB1}
Let $1 < a < b < c < e < d-1$ be a sequence of integers with $b +1 < c$ and 
$R = k\big[x^\a y^{d-\a}|\ \a \in \{0,d\} \cup [a,b] \cup [c,e]\big]$. 
Assume that $2a \le b +1$ and $c+d-1 \le 2e$.
Then $R$ is a Buchsbaum ring if and only if 
\begin{align*}
2b+1&\ge c,\\
\max \left\{2b, e\right\}+1&\ge \min \left\{a+c, d\right\},\\
b+e+1&\ge \min \left\{2c, d\right\},\\
\max \left\{2b, d\right\}+1&\ge a+c,\\
\max \left\{d, b+e\right\}+1&\ge \min \left\{a+d, 2c\right\},\\
b+d+1&\ge 2c.
\end{align*}
\end{cor}

\begin{proof}
There are 6 different symmetric posets $J$ of $I$ with $(1,1) \in J$ and $(2,2) \not\in J$:
\begin{enumerate}
\item $\{(0,1),(1,0),(1,1)\}$,
\item $\{(0,1),(1,0),(0,2),(2,0),(1,1)\}$,
\item $\{(0,1),(1,0),(0,2),(2,0),(1,1),(1,2),(2,1)\}$,
\item $\left\{(0, 0), (0, 1), (1, 0), (1, 1), (0, 2), (2, 0), (0, 3), (3, 0)\right\},$
\item $\left\{(0, 0), (0, 1), (1, 0), (1, 1), (0, 2), (2, 0), (0, 3), (3, 0), (1, 2), (2, 1)\right\},$
\item $\left\{(0, 0), (0, 1), (1, 0), (1, 1), (0, 2), (2, 0), (0, 3), (3, 0), (1, 2), (2, 1), (1, 3), (3, 1)\right\}.$
\end{enumerate}
By (2.1), these posets yield the 6 inequalities in the statement.
\end{proof}

\begin{ex}
Let $R = k[x^d,x^{d-2}y^2,x^{d-3}y^3,x^3y^{d-3},x^2y^{d-2},y^d]$, $d \ge 8$. Then $a=2$, $b=3$, $c=d-3$, and $e = d-2$, which satisfies the assumption of Corollary \ref{BmB1}. Using Corollary \ref{BmB1} we can easily check that $R$ is a Buchsbaum ring if and only if $d = 8,9, 10$. 
\end{ex}


\section{Regularity estimates}

We adhere to the notations of the preceding sections. Let $M$ be a set of monomials of degree $d$ in $k[x,y]$ and $R = k[M]$. Let $0=a_0 \le  a_1 \le \cdots \le a_{2r+1} = d$ be a sequence of integers with $a_{2i-1} < a_{2i}-1$, $i = 1,...,r$, such that
$$G_M =  \bigcup_{i=0}^r[a_{2i}, a_{2i+1}].$$
The integer gaps between the intervals of $G_M$ are the integer intervals $[a_{2i-1}+1,a_{2i}-1]$, whose length is $\ell_i := a_{2i} - a_{2i-1}-1$, $i = 1,...,r$. The aim of this section is to estimate $\reg(R)$ in terms of the maximal length of integer gaps.

By the Goto-Eisenbud regularity conjecture \cite{EG}, which was proved by Gruson, R. Lazarsfeld and C. Peskine \cite{GLP} for curves, we know that 
$\reg(R) \le e(R) - \text{codim}(R)$, where $e(R)$ denotes the multiplicity of $R$. In general, this conjecture  does not hold as shown recently by McCullough and Peeva \cite{MP}.
Since $e(R) \le d$ and
$\text{codim}(R) = |M| -2$, we get $e(R) - \text{codim}(R) \le d - |M|+2$. It is easy to see that $d - |M|+1 = \sum_{i=1}^r \ell_i$. Therefore, $\reg(R) \le \sum_{i=1}^r \ell_i+1$. However, this bound is far from the best possible. 
L'vosky \cite[Proposition 5.5]{Lv} showed that $\reg(R) \le \ell +\ell'+1$, where $\ell$ and $\ell'$ denote the largest and second largest length of the gaps. If there is only a gap, we set $\ell' = 0$. 

In the smooth case $a_1 > 0$ and $a_{2r} < d$,
Hellus, Hoa and St\"uckrad \cite[Theorem 2.7]{HHS} gave the bound
$$\reg(R) \le \left\lfloor \dfrac{\ell-1}{\e}\right\rfloor +2,$$ 
where $\e = \min\{a_1, d-a_{2r}\}$, which is only a fraction of the bound of L'vosky. 
They also presented some cases where this bound is attained.  By the proof of Theorem \ref{Bm},
$\reg(R) = \min\{n\ge 1|\ nG_M = [0,nd]\}$ in this case. This fact was observed already in \cite[Corollary 2.5(ii)]{HHS}, where it was used to prove the above bound.

For non-smooth curves of Types A and B we have shown in the proofs of Theorems \ref{BmA} and \ref{BmB} that the finite Macaulayfication of $R$ is generated by a set $N$ of monomials of degree $d$ and that $\reg(R) = \min\{n \ge 2|\ nG_M = nG_N\}.$ 
To estimate $nG_M$ and $nG_N$ we shall need the following observations. The reader may skip this technical part and goes directly to the main results of this section, which establish regularity bounds for non-smooth curves of Types A and B, starting with Theorem \ref{regA}.

\begin{lem} \label{fraction1}
Let $0 \le a < b$ be integers such that $2a - 1\le b$.
Let $H = \{0\} \cup [a,b]$. Then
\begin{enumerate}[\indent \rm (1)]
\item $[a,nb] \subseteq nH$ for all $n \ge 1$.
\item For all integers $c > \a > b$, $\a + nb \in nH+c$ for $n \ge \left\lfloor \dfrac{a + c - \a - 1}{b}\right\rfloor+1$.
\end{enumerate}
\end{lem}

\begin{proof}
(1) Since $2a - 1\le b$, we have $ib \ge (i+1)a-1$ for all $i \ge 1$. From this it follows that
$[ia,ib] \cup [(i+1)a,(i+1)b] = [ia,(i+1)b]$. Therefore,
$$[a,nb] = [a,b] \cup [2a,2b] \cup \cdots \cup [na,nb] = 
\bigcup_{i=1}^n i[a,b] = \bigcup_{i=1}^n ((t-i)0 +i[a,b]) \subset nH.$$

(2) We have $a + c - \a -1 < nb$. Hence, $a \le \a + nb - c \le nb$. This implies $\a + nb - c \in [a,nb] \subseteq nH$. Therefore, $\a + nb \in nH + c$.
\end{proof}

\begin{lem} \label{fraction2}
Let $0 < c < e \le d$ be integers such that $d+c -1 \le 2e$.
Let $H = [c,e]  \cup \{d\}$. Then
\begin{enumerate}[\indent \rm (1)]
\item $[nc,e+(n-1)d] \subseteq nH$ for all $n \ge 1$.
\item For all integers $b <  \a < c$, $\a + nc \in nH + b$ for $n \ge \left\lfloor \dfrac{d- e - b + \a - 1}{d-c}\right\rfloor+1$.
\end{enumerate}
\end{lem}

\begin{proof}
Let $a' = d-e$ and $b' = d-c$. Then $2a' - 1 = 2(d-e)-1 \le 2d - (d+c) = b'$.
Let $H' = \{0\} \cup [a',b']$. Then $nH = \{nd-\a|\ \a \in nH'\}$.
By Lemma \ref{fraction1}(1), $[a',nb'] \subseteq nH'$ for all $n \ge 1$. 
Therefore, $[nc,e+(n-1)d] \subseteq nH$. 

For all integers $b < \a < c$, we have $d-b > d-\a > b'$. 
By Lemma \ref{fraction1}(2), 
$b - \a + nb' = (d-\a) + nb' - (d-b) \in nH'$ for $n \ge \left\lfloor \dfrac{a' - b + \a - 1}{b'}\right\rfloor+1$.
Thus, $nc - b +\a = nd - (b-\a + nb') \in nH$. 
Hence, $\a + nc \in nH+b$ for $n \ge \left\lfloor \dfrac{d- e - b + \a - 1}{d-c}\right\rfloor+1$.
\end{proof}

\begin{lem} \label{fraction3}
Let $0 < b < c < d$ be integers. Let $H = [0,b] \cup [c,d]$. Then 
\begin{enumerate}[\indent \rm (1)]
\item $nb + 1 \not\in nH$ for $n \le \left\lfloor\dfrac{c - 2}{b}\right\rfloor$.
\item $nc - 1 \not\in nH$ for $n \le  \left\lfloor \dfrac{d - b - 2}{d-c}\right\rfloor$.
\end{enumerate}
\end{lem}

\begin{proof}
For $n \le \dfrac{c - 2}{b}$ we have $nb + 1 \le c-1$. If $nb+1 \in nH$, then $nb+1$ is a sum of $n$ elements of $H$. These elements can not be all $\le b$.
Hence, one of them is $> b$. Since $c$ is the least element $> b$ of $H$, this implies $nb+1 \ge c$, a contradiction. This proves (1). Applying (1) to the sequence $0 < d-c < d-b < d$, we obtain (2).
\end{proof}

We have the following regularity bounds for non-smooth curves of Type A.

\begin{thm} \label{regA}
Let $0=a_0 \le a_1 \le \cdots \le a_{2r+1}=d$ be a sequence of integers with $a_{2i-1} < a_{2i}-1$, $i = 1,...,r$, and $R = k\big[x^\a y^{d-\a}|\ \a \in \bigcup_{i=0}^r[a_{2i}, a_{2i+1}]\big]$. 
Assume that $a_1 = 0$, $a_{2r} < d$, and $2a_2 - 1\le a_3$, $r \ge 2$. 
Set $\ell = \max\{a_{2i} - a_{2i-1}-1|\ i = 2,...,r\}$ and $\e = \min\{a_3,d-a_{2r}\}$.  Then 
\begin{enumerate}[\indent \rm (1)]
\item $\reg(R) \le  \left\lfloor \dfrac{\ell -1}{\e}\right\rfloor+3$ if $a_3-a_2 < d-a_{2r}$.
\item $\reg(R) \le  \left\lfloor \dfrac{\ell -1}{\e}\right\rfloor+2$ if $a_3-a_2 \ge d-a_{2r}$.
\item $\reg(R) \ge \max\left\{\left\lfloor\dfrac{a_{4}-2}{a_3}\right\rfloor,\left\lfloor\dfrac{d-a_{2r-1}-2}{d-a_{2r}}\right\rfloor\right\}+1.$
\end{enumerate}
\end{thm}

\begin{proof}
Let $M = \{x^\a y^{d-\a}|\ \a \in \bigcup_{i=0}^r[a_{2i}, a_{2i+1}]\}$ and $N = \{x^{\a}y^{d-\a}| \ \a \in \{0\} \cup [a_2,d]\}$. By the proof of Theorem \ref{BmA}, 
$\reg(R) =  \min\{n \ge 2|\ nG_M = nG_N\}.$

We have $nG_N = \bigcup_{i=0}^ni[a_2,d]$ for $n \ge 1$.
From this it follows that $nG_N \subseteq \{0\} \cup [a_2,nd]$.
Since $d > a_3 \ge 2a_2-1$, $[a_2,nd] \subseteq nG_N$ by Lemma \ref{fraction1}(1).
Hence $nG_N =  \{0\} \cup [a_2,nd]$. We also have
$[a_2,na_3] \subseteq nG_M$ by Lemma \ref{fraction1}(1) and $[na_{2r},nd] \subseteq nG_M$ by Lemma  \ref{fraction2}(1) (case $e=d$). 
Therefore, 
$$nG_N \setminus nG_M \subseteq [a_2,nd] \setminus ([a_2,na_3] \cup [na_{2r},nd]) \subseteq [na_3,na_{2r}].$$
Hence,  $nG_N = nG_M$ if and only if $[na_3,na_{2r}] \subseteq nG_M$. This implies
\begin{equation}
\reg(R) = \min\{n|\ [na_3,na_{2r}] \subseteq nG_M\}.
\end{equation}

Let $m = \reg(R)-1$. Then there exists $\nu \in [ma_3,ma_{2r}]$ such that $\nu \not\in mG_M$.
We have 
$$[2a_3,2a_{2r}] = [2a_3,a_3+a_{2r}] \cup [a_3+a_{2r},2a_{2r}] = [a_3,a_{2r}] + \{a_3,a_{2r}\}.$$
Using induction we can show that
$[na_3,na_{2r}] = [a_3,a_{2r}] + (n-1)\{a_3,a_{2r}\}$ for $n \ge 2$.
Therefore, there exists $\a \in  [a_3,a_{2r}]$ such that 
$\nu = \a + pa_3 + qa_{2r}$ for some $p,q \ge 0$, $p + q = m-1$.
Since $\nu \not\in mG_M$, $\a \not\in G_M$. 
The condition $\nu \not\in mG_M$ also implies $\a + pa_3 \not\in (p+1)G_M$ and $\a +qa_{2r} \not\in (q+1)G_M$.

Since $\a \not\in G_M$, there exists $i$ such that $a_{2i-1} < \a < a_{2i}$, $1 < i \le r$.
Let $H = \{0\} \cup [a_2,a_3]$. Since $H \subseteq G_M$ and $a_{2i} \in G_M$, $\a + pa_3 \not\in pH + a_{2i}$.
By Lemma \ref{fraction1}(2), this implies 
$$p \le \left\lfloor \dfrac{a_2 + a_{2i} - \a - 1}{a_3}\right\rfloor \le \left\lfloor \dfrac{a_2 + a_{2i} - \a - 1}{\e}\right\rfloor.$$
Using Lemma \ref{fraction2}(2) (case $e = d$), we also have 
$$q \le \left\lfloor \dfrac{\a - a_{2i-1} - 1}{d-a_{2r}}\right\rfloor \le \left\lfloor \dfrac{\a - a_{2i-1} - 1}{\e}\right\rfloor .$$
From this it follows that
$$
p+q \le \dfrac{a_2 + a_{2i} - \a - 1}{\e} + \dfrac{\a-a_{2i-1} - 1}{\e}
= \dfrac{a_{2i} - a_{2i-1} + a_2 - 2}{\e} \le \dfrac{\ell + a_2 - 1}{\e}.
$$
Therefore,
$\reg(R) = p+q+2 \le \left\lfloor \dfrac{\ell + a_2 - 1}{\e}\right\rfloor + 2.$

This bound implies (1). Indeed, if $\e = a_3$, we have
$$\reg(R) \le \left\lfloor \dfrac{\ell + a_2 - 1}{a_3}\right\rfloor + 2 \le \left\lfloor \dfrac{\ell - 1}{a_3}\right\rfloor + 3$$
because $a_2 \le a_3$. 
If $\e = d-a_{2r}$ and $a_3-a_2 < d-a_{2r}$, we have $a_2 \le a_3-a_2 + 1 \le d-a_{2r}$. Hence
$$\reg(R) \le \left\lfloor \dfrac{\ell + a_2 - 1}{d-a_{2r}}\right\rfloor + 2 \le \left\lfloor \dfrac{\ell - 1}{d-a_{2r}}\right\rfloor + 3.$$
\par

If $a_3-a_2 \ge d-a_{2r}$, then $\e = d-a_{2r}$ and $a_{2i} - \a-1 < n\e \le n(a_3-a_2)$ for $n \ge \left\lfloor \dfrac{a_{2i} - \a - 1}{\e}\right\rfloor+1$. From this it follows that $na_2 \le \a + na_3 - a_{2i} \le na_3$. Hence $\a + na_3 - a_{2i} \in [na_2,na_3] \subseteq nG_M$,
which implies $\a + na_3 \in nG_M + a_{2i} \subseteq (n+1)G_M$. Therefore, $p \le  \left\lfloor \dfrac{a_{2i} - \a - 1}{\e}\right\rfloor$. Now we have
$$
p+q \le \dfrac{a_{2i} - \a - 1}{\e} + \dfrac{\a - a_{2i-1} - 1}{\e} 
= \dfrac{a_{2i} - a_{2i-1} - 2}{\e} \le \dfrac{\ell -1}{\e}.
$$
Therefore,
$\reg(R) = p+q+ 2 \le \left\lfloor \dfrac{\ell - 1}{\e}\right\rfloor +2$, which proves (2).
\par

Now we are going to prove (3). By Lemma \ref{fraction3}(1), $na_3+1 \not\in nG_M$ for $n \le \left\lfloor \dfrac{a_4 -2}{a_3}\right\rfloor$.
From this it follows that $[na_3,na_{2r}] \not\subseteq nG_M$. Hence
$\reg(R) \ge \left\lfloor \dfrac{a_4 -2}{a_3}\right\rfloor+1$ by (3.1). 
By Lemma \ref{fraction3}(2), $na_{2r}-1 \not\in nG_M$ for $n \le \left\lfloor \dfrac{d - a_{2r-1} -2}{d-a_{2r}}\right\rfloor$. From this it follows that $[na_3,na_{2r}] \not\subseteq nG_M$. Hence 
$\reg(R) \ge \left\lfloor \dfrac{d - a_{2r-1} -2}{d-a_{2r}}\right\rfloor+1$, as required.
\end{proof}

\begin{rem} \label{remA}
The bound of Theorem \ref{regA}(1) is sharp. Actually, we have 
$$\reg(R) \le \left\lfloor \dfrac{\ell + a_2 - 1}{\e}\right\rfloor + 2$$
by the proof of Theorem \ref{regA}(1). This bound is attained
if $\ell = a_4-a_3-1$, $\e = a_3$ and $a_5 < a_4+a_2-1 < a_6$ for $r \ge 3$.
Indeed, we have $(n-1)a_3 \le a_4 - a_3 +a_2 - 2$ for $n \le \left\lfloor \dfrac{\ell + a_2 -1}{\e}\right\rfloor+1$. From this it follows that $na_3 + 1 \le a_4+a_2-1$. Let $\a := a_4+a_2-1$.
If $\a \in nG_M$, then $\a$ is a sum of $n$ elements of $G_M$. 
These elements can not be all $\le a_3$. Hence one of them is $\ge a_4$. 
The remaining elements have to be zero because all non-zero elements of $G_M$ are $\ge a_2$. Therefore, $\a \in G_M$. Since $\a \in [a_5+1,a_6-1]$, which does not contain any element of $G_M$, this gives a contradiction. So we have $\a \not\in nG_M$. 
Since $na_3 < \a < na_6 \le na_{2r}$, $[na_3,na_{2r}] \not\subseteq nG_M$. 
By (3.1), this implies $\reg(R) \ge \left\lfloor \dfrac{\ell + a_2 -1}{\e}\right\rfloor+2$. Using the above conditions one can construct examples where
$$\reg(R) = \left\lfloor \dfrac{\ell +a_2-1}{\e}\right\rfloor+2 = \left\lfloor \dfrac{\ell -1}{\e}\right\rfloor+3.$$
\end{rem}

\begin{ex} 
Let $R = k[x^d,x^{d-1}y,x^{d-2}y^2, x^{d-3}y^3, x^{d-5}y^5, x^3y^{d-3}, x^2y^{d-2},y^d]$, $d \ge 10$. Then $R$ belongs to Type $A$ with $r = 3$, $a_1 = 0$, $a_2 = 2$, $a_3 = 3$, $a_4 = a_5 = d-5$, $a_6 = d-3$, $a_7 = d$, which satisfy the condition for the attainment of the regularity bound of Remark \ref{remA}. In this case, $\ell = d-9$ and $\e = 3$.
Therefore, $\reg(R) = \left\lfloor \dfrac{d-8}{3}\right\rfloor+2$. If $d-10$ is not divisible by 3, $\left\lfloor \dfrac{d-8}{3}\right\rfloor+2 = \left\lfloor \dfrac{d-10}{3}\right\rfloor+3.$
\end{ex}

If $r = 2$, we have the following explicit regularity formula.

\begin{thm}  \label{regA1}
Let $1 < a <  b < c  < d$ be a sequence of integers with $b < c-1$ and $R = k\big[x^\a y^{d-\a}|\ \a \in \{0\} \cup [a,b] \cup [c,d]\big]$. Assume that $2a - 1\le b$. Let $\e = \min\{b,d-c\}$.
Then $\reg(R) =  \left \lfloor \dfrac{c-b-2}{\e}\right \rfloor +2.$
\end{thm}

\begin{proof}
By Theorem \ref{regA}(3), 
\begin{align*}
\reg(R) & \ge \max\left\{\left\lfloor\dfrac{c-2}{b}\right\rfloor,\left\lfloor\dfrac{d-b-2}{d-c}\right\rfloor\right\}+1\\
& = \max\left\{\left\lfloor\dfrac{c-b-2}{b}\right\rfloor,\left\lfloor\dfrac{c-b-2}{d-c}\right\rfloor\right\}+2
= \left \lfloor \dfrac{c-b-2}{\e} \right \rfloor + 2.
\end{align*}
It remains to show that $\reg(R) \le \left \lfloor \dfrac{c-b-2}{\e} \right \rfloor + 2$.
By Theorem \ref{regA}(2), this bound holds if $b-a \ge d-c$.
Thus, we may assume that $b-a < d-c$.
By the proof of Theorem \ref{regA}, we have 
$$\reg(R) = \min\{n|\ [nb,nc] \subseteq nG_M\},$$
where $G_M = \{0\} \cup [a,b] \cup [c,d]$.

Let $m = \reg(R)-1$ and $\nu \in [mb,mc]$ such that $\nu \not\in mG_M$. Since $[mb,mc] = [b,c] + (m-1)\{b,c\}$, there exists $\a \in  [b,c]$ such that 
$\nu = \a + pb + qc$ for some $p,q \ge 0$, $p + q = m-1$.
The condition $\nu \not\in mG_M$ implies $\a + pb \not\in (p+1)G_M$ and $\a + qc \not\in (q+1)G_M$. 

Let $t = \left \lfloor \dfrac{c-\a-1}{b} \right \rfloor+1$. Then  $c-\a \le  tb  \le c-\a +b -1$. Hence $\a + tb \in [c,c+b]$. Since $2a-1 \le b$ and $b-a < d-c$, we have $c + a - 1 \le c+b-a < d$. This implies $[c,c+b] = [c,d] \cup [c+a,c+b] \subseteq 2G_M \subseteq (t+1)G_M$. Therefore, $\a + tb \in (t+1)G_M$. 
From this it follows that $\a + nb \in (n+1)G_M$ for $n \ge t$. Hence
$$p \le  \left \lfloor \dfrac{c-\a-1}{b} \right \rfloor \le \left \lfloor \dfrac{c-\a-1}{\e} \right \rfloor.$$
By Lemma \ref{fraction2}(2) (case $e=d$), $\a + nc \in (n+1)G_M$ for 
$n \ge \left\lfloor \dfrac{\a-b-1}{d-c} \right\rfloor+1$. This implies 
$$q \le \left\lfloor \dfrac{\a-b-1}{d-c} \right\rfloor   \le \left \lfloor \dfrac{\a-b-1}{\e}\right \rfloor.$$ 
Therefore,
$$p + q \le \dfrac{c-\a-1}{\e}  + \dfrac{\a-b-1}{\e} = \dfrac{c-b-2}{\e}.$$
Since $\reg(R) = p+q+2$, we obtain $\reg(R) \le \left\lfloor \dfrac{c-b-2}{\e} \right\rfloor+2$, as required.
\end{proof}

\begin{ex} 
Let $R = k[x^d,x^{d-1}y,x^3y^{d-3}, x^2y^{d-2},y^d]$, $d \ge 6$. Then $R$ satisfies the assumption of Theorem \ref{regA1} with $r = 2$, $a = 2$, $b = 3$, $c = d-1$. Therefore, $\reg(R) = d-4$. 
\end{ex}

In the following, we give regularity bounds for non-smooth curves of Type B. 

\begin{thm} \label{regB}
Let $0=a_0 \le a_1 \le \cdots \le a_{2r+1}=d$ be a sequence of integers with $a_{2i-1} < a_{2i}-1$, $i = 1,...,r$, and $R = k\big[x^\a y^{d-\a}|\ \a \in \bigcup_{i=0}^r[a_{2i}, a_{2i+1}]\big]$. 
Assume  that $a_1 = 0$, $a_{2r} = d$, $2a_2 - 1\le a_3$,  and $a_{2r-2} +d-1 \le 2a_{2r-1}$, $r \ge 3$. 
Set $\ell = \max\{a_{2i} - a_{2i-1}-1|\ i = 2,...,r-1\}$ and $\e = \min\{a_3,d-a_{2r-2}\}$. 
Then
\begin{enumerate}[\indent \rm (1)]
\item $\reg(R) \le \left\lfloor \dfrac{\ell+a_2 + d-a_{2r-1}-1}{\e}\right\rfloor+2$.
\item $\reg(R) \le \left\lfloor \dfrac{\ell+ d-a_{2r-1}-1}{\e}\right\rfloor+2$ if $a_3 -a_2 \ge d-a_{2r-2}$.
\item $\reg(R) \le \left\lfloor \dfrac{\ell+a_2-1}{\e}\right\rfloor+2$ if $a_{2r-1}-a_{2r-2} \ge a_3$. 
\item $\reg(R) \ge \max\left\{\left\lfloor\dfrac{a_{4}-2}{a_3}\right\rfloor,\left\lfloor\dfrac{d-a_{2r-3}-2}{d-a_{2r-2}}\right\rfloor\right\}+1.$
\end{enumerate}
\end{thm}

\begin{proof}
Let $M = \{x^\a y^{d-\a}|\ \a \in \bigcup_{i=0}^r[a_{2i}, a_{2i+1}]\}$ and $N = \{x^{\a}y^{d-\a}| \ \a \in \{0,d\} \cup [a_2,a_{2r-1}]\}$. By the proof of Theorem \ref{BmB}, 
$\reg(R) =  \min\{n \ge 2|\ nG_M = nG_N\}.$

We have $nG_N = \bigcup_{i+j \le n}[ia_2+jd,ia_{2r-1}+jd]$ for $n \ge 1$.
From this it follows that $nG_N \subseteq \{0,nd\} \cup [a_2,a_{2r-1}+(n-1)d]$.
Since $a_{2r-1} > a_3 \ge  2a_2-1$, $[a_2,na_{2r-1}] \subseteq nG_N$ by Lemma \ref{fraction1}(1).
Since $a_2+ d-1 < a_{2r-2} +d-1 \le 2a_{2r-1}$, $[na_2,a_{2r-1}+(n-1)d] \subseteq nG_N$ by Lemma \ref{fraction2}(1). Hence 
$$nG_N \supseteq [a_2,na_{2r-1}] \cup [na_2,a_{2r-1}+(n-1)d] =  [a_2,a_{2r-1}+(n-1)d].$$
Therefore, $nG_N = \{0,nd\} \cup [a_2,a_{2r-1}+(n-1)d].$

By Lemma \ref{fraction1}(1), $[a_2,na_3] \subseteq nG_M$. 
By Lemma \ref{fraction2}(1), $[na_{2r-2},a_{2r-1}+(n-1)d] \subseteq nG_M$. 
Therefore, 
\begin{align*}
nG_N \setminus nG_M & \subseteq   [a_2,a_{2r-1}+(n-1)d] \setminus ([a_2,na_3] \cup [na_{2r-2},a_{2r-1}+(n-1)d])\\
& \subseteq [na_3,na_{2r-2}].
\end{align*}
From this it follows that $nG_N = nG_M$ if and only if $[na_3,na_{2r-2}] \subseteq nG_M$. So we have
\begin{equation}
\reg(R) = \min\{n \ge 2|\ [na_3,na_{2r-2}] \subseteq nG_M\}.
\end{equation}

Now we can proceed as in the proof of Theorem \ref{regA}.
Let $m = \reg(R)-1$ and $\nu \in [ma_3,ma_{2r-2}]$ such that $\nu \not\in mG_M$.
Then there exists $\a \in  [a_3,a_{2r-2}]$ such that 
$\nu = \a + pa_3 + qa_{2r-2}$ for some $p,q \ge 0$, $p + q = m-1$. 
Since $\nu \not\in mG_M$, $\a \not\in G_M$.  
Hence $a_{2i-1} < \a < a_{2i}$ for some $i$, $1 < i < r$.
By the proof of Theorem \ref{regA}(1), we have
$$p \le \left\lfloor \dfrac{a_2 + a_{2i} - \a - 1}{\e}\right\rfloor.$$
Since $d-a_{2i} < d- \a < d-a_{2i-1}$, we can replace $p,a_2,a_{2i}, \a$ by $q,d-a_{2r-1},d-a_{2i-1},d-\a$ to obtain
\begin{equation*}
q  \le \left\lfloor \dfrac{\a - a_{2i-1} + d-a_{2r-1} - 1}{\e}\right\rfloor.
\end{equation*}
From this it follows that
\begin{align*}
p+q & \le \dfrac{a_2 + a_{2i} - \a - 1}{\e} + \dfrac{\a-a_{2i-1} + d-a_{2r-1} - 1}{\e}\\
& = \dfrac{a_{2i} - a_{2i-1} + a_2 + d-a_{2r-1}- 2}{\e} \le \dfrac{\ell + a_2 + d-a_{2r-1} - 1}{\e}.
\end{align*}
Since $\reg(R) = p+q+2$, we obtain
$\reg(R) \le \left\lfloor \dfrac{\ell + a_2 + d-a_{2r-1} - 1}{\e}\right\rfloor + 2,$
which proves (1).
\par

If $a_3-a_2 \ge d-a_{2r-2}$, we can show similarly as in the proof for Theorem \ref{regA}(2) that 
$p \le  \left\lfloor \dfrac{a_{2i} - \a - 1}{\e}\right\rfloor$. Therefore,
\begin{align*}
p+q & \le \dfrac{a_{2i} - \a - 1}{\e} + \dfrac{\a - a_{2i-1} + d-a_{2r-1} - 1}{\e}\\
& = \dfrac{a_{2i} - a_{2i-1} + d-a_{2r-1} - 2}{\e} \le \dfrac{\ell + d-a_{2r-1}-1}{\e}.
\end{align*}
This implies
$\reg(R) \le \left\lfloor \dfrac{\ell + d-a_{2r-1} - 1}{\e}\right\rfloor +2$, which proves (2).

Note that the assumption of Theorem \ref{regB} is symmetric if we replace the numbers $a_i$ by $d-a_{2r+1-i}$, $i = 0,1,...,2r+1$. 
Since $a_{2r-1}-a_{2r-2} \ge a_3$ is the symmetric condition of $a_3-a_2 \ge d-a_{2r-2}$, we also have (3), 
which is the symmetric statement of (2).\par

To prove (4) we observe that $na_3+1 \not\in nG_M$ for $n \le \left\lfloor \dfrac{a_4 -2}{a_3}\right\rfloor$ by Lemma \ref{fraction3}(1). Since $na_3+1 \in [na_3,na_{2r-2}]$,  $[na_3,na_{2r-2}] \not\subseteq nG_M$. By (3.2), this implies 
$\reg(R) \ge \left\lfloor \dfrac{a_4 -2}{a_3}\right\rfloor+1$. By the duality of the variables $x,y$, we can replace $a_3,a_4$ by $d-a_{2r-2}, d - a_{2r-3}$ to obtain
$\reg(R) \ge \left\lfloor \dfrac{d - a_{2r-3} -2}{d-a_{2r-2}}\right\rfloor+1$. The proof of Theorem \ref{regB} is now complete.
\end{proof}

If $r=3$, we have better regularity estimates than Theorem \ref{regB}. 
						
\begin{thm}  \label{regB1}
Let $1 < a < b < c < e < d-1$ be a sequence of integers with  $b < c-1$ and
$R = k\big[x^\a y^{d-\a}|\ \a \in \{0,d\} \cup [a,b] \cup [c,e] \big]$. 
Assume that $2a -1 \le b$ and $c+d-1 \le 2e$.
Let $\e = \min\{b,d-c\}$. Then
\begin{enumerate}[\indent \rm (1)]
\item $\left \lfloor \dfrac{c-b-2}{\e}\right \rfloor + 2 \le \reg(R) \le \left \lfloor \dfrac{c-b-2}{\e}\right \rfloor + 3$.
\item $\reg(R) = \left \lfloor \dfrac{c-b-2}{\e}\right \rfloor + 2$ if one of the following conditions is satisfied:\par
{\rm (i)} $a-1 \le e-c$ and $d-e-1 \le b-a$, \par
{\rm (ii)}  $b-a \ge d-c$,\par
{\rm (iii)}  $e-c \ge b$.
\end{enumerate}
\end{thm}

\begin{proof}
By the proof of Theorem \ref{regB}, we have 
$\reg(R) = \min\{n|\ [nb,nc] \subseteq nG_M\},$
where $G_M = \{0,d\} \cup [a,b] \cup [c,e]$.

Let $m = \reg(R)-1$ and $\nu \in [mb,mc]$ such that $\nu \not\in mG_M$. 
Since $[mb,mc] = [b,c] + (m-1)\{b,c\}$, there exists $\a \in  [b,c]$ such that 
$\nu = \a + pb + qc$ for some $p,q \ge 0$, $p + q = m-1$.
The condition $\nu \not\in mG_M$ implies $\a + pb \not\in (p+1)G_M$. 

If $a - 1 \le e-c$, then $c+a-1 \le e$.
This implies $[c,c+b] \subseteq [c,e] \cup [c+a,c+b] \subseteq 2G_M$.
Let $t = \left \lfloor \dfrac{c-\a-1}{b} \right \rfloor+1$. Then  
$c-\a \le  tb  \le c+b -\a -1$. Hence $\a + tb \in [c,c+b] \subseteq 2G_M \subseteq (t+1)G_M$. 
Therefore, $\a + tb \in (t+1)G_M$. 
From this it follows that $\a + nb \in (n+1)G_M$ for $n \ge t$. Hence
\begin{equation}
p \le \left \lfloor \dfrac{c-\a-1}{b} \right \rfloor \le \left \lfloor \dfrac{c-\a-1}{\e} \right \rfloor.
\end{equation}
As in the proof of Theorem \ref{regB}(1), we have
$q \le \left \lfloor \dfrac{\a - b + d-e -1}{\e} \right \rfloor.$
Therefore,
$$p+q \le \dfrac{c-\a-1}{\e} + \dfrac{\a-b+d-e-1}{\e} = \dfrac{c-b+d-e-2}{\e}.$$
Since $\reg(R) = p+q+2$,  we obtain
\begin{equation}
\reg(R) \le \left \lfloor \dfrac{c - b +d-e-2}{\e} \right \rfloor + 2.
\end{equation}
By Theorem \ref{regB}(2), this bound also holds if $b - a \ge d-c$.

Note that the assumption of Theorem \ref{regB1} is symmetric if we replace $a,b,c,e$ by $d-e,d-c,d-b,d-a$. Then we may assume without restriction that $\e = d-c$.
Since $d-e < d-c$, (3.4) implies 
$$\reg(R) \le \left\lfloor \dfrac{c - b +d-e-2}{d-c} \right\rfloor + 2 \le \left \lfloor \dfrac{c - b -2}{d-c} \right\rfloor +3.$$
Therefore, to prove the first bound of (1) we may further assume that $a - 1 > e-c$ and $b-a < d-c$.
Note that $a - 1 >  e-c$ is the symmetric condition of $a - 1 \le e-c$. 
Then we can replace $b,c,d-e$ in (3.4) by $d-c,d-b,a$ to obtain
$$
\reg(R) \le \left\lfloor \dfrac{c-b+a-2}{\e} \right\rfloor + 2 = \left\lfloor \dfrac{c-b+a-2}{d-c} \right\rfloor+2.
$$
Since $a \le b-a + 1 \le d-c$, this implies
$\reg(R) \le \left\lfloor \dfrac{c-b-2}{\e} \right\rfloor + 3.$ 

By Theorem \ref{regB}(4), we have
\begin{align*}
\reg(R) & \ge \max\left\{\left\lfloor\dfrac{c-2}{b}\right\rfloor,\left\lfloor\dfrac{d-b-2}{d-c}\right\rfloor\right\}+1\\
& = \max\left\{\left\lfloor\dfrac{c-b-2}{b}\right\rfloor,\left\lfloor\dfrac{c-b-2}{d-c}\right\rfloor\right\}+2
= \left\lfloor \dfrac{c-b-2}{\e}\right \rfloor + 2.
\end{align*}
This proves the second bound of (1). To prove (2) we only need to show that $\reg(R) \le  \left\lfloor \dfrac{c-b-2}{\e}\right \rfloor + 2$ if one of the conditions (i)-(iii) is satisfied.

If $d-e-1 \le b-a$, we can replace $p,c,\a$ in (3.3) by $q,d-b,d-\a$ to obtain $q \le \left \lfloor \dfrac{\a - b -1}{\e} \right \rfloor$. Therefore, if $a - 1 \le e-c$ and $d-e-1 \le b-a$, we have
$$p + q \le \dfrac{c-\a-1}{\e}  + \dfrac{\a-b-1}{\e} = \dfrac{c-b-2}{\e}.$$
Hence, $\reg(R) = p+q+2 \le  \left\lfloor \dfrac{c-b-2}{\e}\right \rfloor + 2.$ 

If $b-a \ge d-c$, then $\e = d-c$ and  $c - \a-1 < n\e \le n(b-a)$ for $n \ge \left\lfloor \dfrac{c - \a - 1}{\e}\right\rfloor+1$. From this it follows that $na \le \a + nb - c \le nb$. Hence $\a + nb - c \in [na,nb] \subseteq nG_M$. Thus, $\a + nb \in nG_M + c \subseteq (n+1)G_M$. Therefore, the condition $\a + pb \not\in (p+1)G_M$ implies $p \le  \left\lfloor \dfrac{c- \a - 1}{\e}\right\rfloor$. 
Since $d-e-1 \le d-c \le b-a < b$, we have $q \le \left \lfloor \dfrac{\a - b -1}{\e} \right \rfloor$ as shown in the preceding paragraph. Therefore,
$$p + q \le \dfrac{c-\a-1}{\e}  + \dfrac{\a-b-1}{\e} = \dfrac{c-b-2}{\e},$$
which implies $\reg(R) \le  \left\lfloor \dfrac{c-b-2}{\e}\right \rfloor + 2.$ 

The symmetric condition of $b-a \ge d-c$ is $e-c \ge b$. Replacing $b,c$ by $d-c,d-b$ in the above bound, we also have $\reg(R) \le  \left\lfloor \dfrac{c-b-2}{\e}\right \rfloor + 2$ if $e-c \ge b$. This concludes the proof of (2).
\end{proof}

Condition (i) of Theorem \ref{regB1}(2) is satisfied if $b-a = e-c$. In fact, the assumption $2a-1 \le b$ and $c+d-1 \le 2e$ implies $a-1 \le b-a = e-c$ and $d-e-1 \le e-c = b-a$.

\begin{ex}
Let $R = k[x^d,x^{d-2}y^2,x^{d-3}y^3,x^3y^{d-3},x^2y^{d-2},y^d]$, $d \ge 8$. Then 
$a = 2$, $b = 3$, $c = d-3$, $e = d-2$, which satisfy the assumption of Theorem \ref{regB1}. 
Since $b-a = c-e$, $\reg(R) = \left \lfloor \dfrac{c-b-2}{\e}\right \rfloor + 2 = \left \lfloor \dfrac{d-2}{3}\right \rfloor.$
\end{ex}

The bound $\reg(R) \le \left \lfloor \dfrac{c-b-2}{\e}\right \rfloor + 3$ of Theorem \ref{regB1}(1) can be attained. 

\begin{ex}
Let $R = k[x^{12},x^{10}y^2,x^9y^3,x^5y^7,x^4y^8,x^3y^9,y^{12}]$. Then $a = 3$, $b = 5$, $c = 9$, $e = 10$, $d=12$, which satisfy the assumption of Theorem \ref{regB1}. Hence $\reg(R) \le \left \lfloor \dfrac{c-b-2}{\e}\right \rfloor + 3 = 3$. We claim that $\reg(R) =  3$. By the proof of 
Theorem \ref{regB},  $\reg(R) = \min\{n \ge 2|\ [5n,9n] \subseteq nG_M\}$, where $G_M = \{0,3,4,5,9,10,12\}.$ It is easy to see that $11 \not\in 2G_M$. Hence, $[10,18] \not\subseteq 2G_M$. From this it follows that $\reg(R) \ge 3$, as required.
\end{ex}

\noindent {\bf Acknowledgement}. This paper is partially supported by grant 101.04-2019.313 of Vietnam National Foundation for Science and Technology Development.


\end{document}